\documentclass[
submission
]{dmtcs-episciences}

\usepackage[utf8]{inputenc}
\usepackage{subfigure}
\usepackage{amsmath,amssymb,amsthm}
\usepackage{epsfig}
\usepackage{color}
\usepackage{enumerate}
\usepackage{algorithm}
\usepackage{algorithmic}
\usepackage{hyperref}
\usepackage[T1]{fontenc}

\newtheorem{theorem}{Theorem}[section]

\newtheorem{corollary}[theorem]{Corollary}

\newtheorem{proposition}[theorem]{Proposition}

\theoremstyle{definition}

\theoremstyle{remark}

\author{Doost Ali Mojdeh\affiliationmark{1}
  \and Iztok Peterin\affiliationmark{2,3}\thanks{Partially supported by the Slovenian research agency under the grants P1-0297, J1-1693 and J1-9109.}
  \and Babak Samadi\affiliationmark{1}\thanks{This work was done during the visit of the University of Cadiz, supported by University of Mazandaran and Ministry of Science, Research and Technology of Iran, under the program Ph.D Research Opportunity.}
	\and Ismael G. Yero\affiliationmark{4}}
\title[(Open) packing number of some graph products]{(Open) packing number of some graph products}
\affiliation{
  Department of Mathematics, University of Mazandaran, Babolsar, Iran\\
  University of Maribor, FEECS, Maribor, Slovenia\\
	IMFM, Ljubljana, Slovenia\\
  Universidad de C\'{a}diz, Algeciras, Spain}
\received{2019-1-23}
\revised{2020-6-19}
\accepted{2020-6-22}
\begin{document}
\publicationdetails{22}{2020}{4}{1}{5100}
\maketitle
\begin{abstract}
The packing number of a graph $G$ is the maximum number of closed neighborhoods of vertices in $G$ with pairwise empty intersections. Similarly, the open packing number of $G$ is the maximum number of open neighborhoods in $G$ with pairwise empty intersections. We consider the packing and open packing numbers on graph products. In particular we give a complete solution with respect to some properties of factors in the case of lexicographic and rooted products. For Cartesian, strong and direct products, we present several lower and upper bounds on these parameters.
\end{abstract}

\section{Introduction}

The packing number $\rho(G)$ and the open packing number $\rho_o(G)$ of a graph $G$ are natural lower bounds of the domination number $\gamma(G)$ and the total domination number $\gamma_t(G)$ of $G$, respectively.  One of the first results of that type is from Meir and Moon \cite{MeMo}, where it was shown that $\rho(T)=\gamma(T)$ for every tree $T$ (in a different notation). It is easy to see that, while the numbers are the same, the sets that yield both $\rho(T)$ and $\gamma(T)$ are often different. See also \cite{b,hr} for some results of that type where they ``unfortunately'' only lie in a shadow of domination and total domination.

In the last decade the packing number became more interesting for itself, and not only in connection with the domination number. Some interesting examples are as follows. The relationship between the packing number and the maximal packings of minimum cardinality, called also the lower packing number, is investigated in \cite{hsa0}. In \cite{MSKG}, a connection between the packing number and the double domination in the form of an upper bound is presented. Graphs for which their packing number equals the packing number of their complement are described in \cite{Dutt}. In \cite{HeLoRa}, it was shown that the domination number can also be  bounded from above by the packing number multiplied by the maximum degree of a graph. Asymptotic bounds for the maximum and the minimum number of packings in graphs of fixed order are established in \cite{JSRz}. A different approach was taken in \cite{BoPe}, where graphs with unique maximum packing are treated.

A generalization of packings presented in \cite{gghr} is called $k$-limited packing, where the closed neighborhood of every vertex can have at most $k$ vertices in a $k$-limited packing set $S$. They exhibited some real-world applications of it to network security, market saturation, NIMBY and codes. A probabilistic approach to some bounds of the $k$-limited packings can be found in \cite{GaZv}. More results on this topic can be found in \cite{msh}. A further generalization, that is, generalized limited packing of the $k$-limited packing (see \cite{DoHiLe}) brings a dynamic approach with respect to the vertices of $G$, where every $v\in V(G)$ can have a different number of neighbors $k_v$ in a generalized limited packing. The problem is NP-complete, but solvable in polynomial time for $P_4$-tidy graphs as shown in \cite{DoHiLe}.

The open packing number was introduced by Henning and Slater in \cite{hs}. They presented theoretical and computational results concerning this parameter of graphs. In \cite{Rall}, the open packing and the total domination equality $\rho_{o}(T)=\gamma_{t}(T)$ was proved for all trees $T$ of order at least two. Some applications of this equality for trees can be found in \cite{Rall} and \cite{s}. As a generalization of the open packing, and a total version of the limited packing, the concept of total limited packing was introduced in \cite{hms}. A subset $S$ of the vertices is called a $k$-total limited packing if the open neighborhood of each vertex has at most $k$ neighbors in $S$.

The class of graphs with $\rho(G)=\gamma(G)$, where both maximum packing sets and minimum dominating sets coincide, is called efficient closed domination graphs. In such a case we also call a minimum dominating set a 1-perfect code. Similarly, we call a graph $G$ an efficient open domination graph, whenever $\rho_o(G)=\gamma_t(G)$. Efficient open domination graphs are well known among graph products, see \cite{KPRT,KuPeYe1}. In particular, in \cite{KPRT}, a method was established concerning how to approach to the efficient open domination graphs among Cartesian products where one factor is fixed. The study of perfect codes in graphs was initiated by Biggs~\cite{biggs-1973}. It was later intensively studied, and for instance, graphs that are both efficient open and efficient closed domination graphs at the same time are discussed in \cite{KPY}.

We twist the roles usually played by (open) packings in investigations, and bring in this work, packing and open packing on the front line and present several upper and lower bounds for the packing and open packing numbers of graph products. In the next section we fix the notation. Then we start with the investigation of packing and open packing numbers of the Cartesian product of graphs. A section that covers the lexicographic, the strong and the direct product follows. We end with a section on the rooted product of graphs.


\section{Preliminaries}

Throughout this paper, let $G$ be a finite simple graph with vertex set $V(G)$ and edge set $E(G)$.
The {\em open neighborhood} $\{u\in V(G): uv\in E(G)\}$ of a vertex $v$ is denoted by $N_G(v)$, and the {\em closed neighborhood} of $v$ is $N_G[v]=N_G(v)\cup \{v\}$. The {\em degree} of a vertex $v$ is $|N_G(v)|$ and is denoted by $\delta_G(v)$. The {\em minimum} and {\em maximum degree} of $G$ are denoted by $\delta(G)$ and $\Delta(G)$, respectively. A subgraph of a graph $G$ induced by $S\subseteq V(G)$ is denoted by $G[S]$.

A set $S\subseteq V(G)$ is a {\em dominating set} if each vertex in $V(G)\backslash S$ has at least one neighbor in $S$. The {\em domination number} $\gamma(G)$  is the minimum cardinality of a dominating set of $G$. We call a dominating set $S$ of a graph $G$ with cardinality $\gamma(G)$ a $\gamma(G)$-set. Similarly, a set $D\subseteq V(G)$ is a {\em total dominating set} if each vertex in $V(G)$ has at least one neighbor in $D$. The {\em total domination number} $\gamma_t(G)$ is the minimum cardinality of a total dominating set of $G$. We call a total dominating set $D$ of a graph $G$ with cardinality $\gamma_t(G)$ a $\gamma_t(G)$-set.

A subset $P\subseteq V(G)$ is a {\em packing} in $G$, if for every pair of distinct vertices $u,v\in P$, $N[u]\cap N[v]=\emptyset$. The {\em packing number} $\rho(G)$ is the maximum cardinality of a packing in $G$. The {\em open packing}, as it is defined in \cite{hs}, is a subset $P\subseteq V(G)$ for which the open neighborhoods of the vertices of $P$ are pairwise disjoint in $G$ (clearly, $P$ is an open packing if and only if $|N(v)\cap P|\leq 1$, for every $v\in V(G)$). The {\em open packing number}, denoted $\rho_{o}(G)$, is the maximum cardinality among all open packings in $G$.

A set $S\subseteq V(G)$ is $k$-{\em independent set} of a graph $G$, if $\Delta (G[S])<k$. The $k$-{\em independence number} of $G$ is the maximum cardinality of a $k$-independent set of $G$.  We denote the $k$-independence number by $\alpha_k(G)$ and call a $k$-independent set of cardinality $\alpha_k(G)$ as an $\alpha_k(G)$-set. This is clearly a generalization of the independence number $\alpha(G)$, which may be considered as the 1-independence  number. We are mainly interested in this work in $\alpha_2(G)$. If $S$ is a 2-independent set of a graph $G$, then $G[S]$ contains isolated edges and isolated vertices. This invariant was first introduced in \cite{FiJa}. For more references we recommend the survey \cite{CFHV}.

For all four standard products of graphs $G$ and $H$ the vertex set of the product is $V(G)\times V(H)$. Their edge sets are defined as follows. In the \emph{Cartesian product} $G\Box H$ two vertices are adjacent if they are adjacent in one coordinate and equal in the other. In the \emph{direct product} $G\times H$ two vertices are adjacent if they are adjacent in both coordinates. The edge set
$E(G\boxtimes H)$ of the \emph{strong product} $G\boxtimes H$ is the union of $E(G\Box H)$ and $E(G\times H)$. Finally, two vertices $(g,h)$ and $(g',h')$ are adjacent in the \emph{lexicographic product}
$G\circ H$ (also $G[H]$) if either $gg'\in E(G)$ or ($g=g'$ and $hh'\in E(H)$). For $h\in V(H)$, $g\in V(G)$ and $\ast \in \{\Box ,\boxtimes ,\times ,\circ \}$, we call $G^{h}=\{(g,h)\in V(G\ast H)\mid g\in V(G)\}$ a $G$-\emph{layer} through $h$ in $G\ast H$, and $^{g}H=\{(g,h)\in V(G\ast H)\mid\ h\in V(H)\}$ an $H$-\emph{layer} through $g$ in $G\ast H$. Note that the subgraph of $G\ast H$ induced by $G^{h}$ is isomorphic to $G$, and the subgraph of $G\ast H$ induced by $^{g}H$ is isomorphic to $H$ for $\ast \in \{\Box ,\boxtimes ,\circ \}$. On the other hand, there are no edges between vertices of $G^{h}$ and between vertices of $^{g}H$ in $G\times H$. Note also that all four products are associative and only the first three are commutative, while the lexicographic product is not, cf.~\cite{ImKl}. The map $p_{G}:V(G\ast H)\rightarrow V(G)$ defined by $p_{G}((g,h))=g$ is called a \emph{projection map onto} $G$ for $\ast \in \{\Box ,\boxtimes ,\times ,\circ \}$. Similarly, we can define the \emph{projection map onto} $H$.


\section{Cartesian product of graphs}

In 1968, Vizing \cite{v} posed the following conjecture, which is still widely open, concerning the domination number of Cartesian product of two graphs $G$ and $H$:
\begin{equation}\label{EQ23}
\gamma(G\square H)\geq \gamma(G)\gamma(H).
\end{equation}
One of the tools which have been used while trying to prove, or making a contribution to the knowledge on the conjecture, is precisely the packing number. This could probably be one of reasons that the packing number has  been very rarely considered by itself in an investigation regarding products of graphs.

One of the most common contributions to the study on Vizing's conjecture concerns finding some Vizing-like results for other domination related parameters in the Cartesian product of graphs, or even in other products. In this sense, similar inequalities concerning many domination parameters have been extensively obtained in the literature (for more information the reader can consult \cite{b}).

We first remark the following Vizing-like inequality for the packing number proved by Kazemi \emph{et al.} \cite{kps}. That is, for all graphs $G$ and $H$,
\begin{equation}\label{EQ21}
\rho(G\square H)\geq \rho(G)\rho(H).
\end{equation}

Although the bound above satisfies the style of a Vizing-like result, in general this bound is not too accurate. That is, the real value for $\rho(G\square H)$ is frequently much larger than $\rho(G)\rho(H)$. For instance, we next show that the difference between $\rho(G\square H)$ and $\rho(G)\rho(H)$ can be arbitrary large. To this end, we first recall that the distance formula in Cartesian product of two graphs $G$ and $H$ is given by
\begin{equation}\label{Dist}
d_{G\square H}((g,h),(g',h'))=d_{G}(g,g')+d_{H}(h,h').\ \ \ ({\rm See}\ \cite{ImKl}.)
\end{equation}

\begin{proposition}
For any integer $b\geq 2$, there exist two graphs $G$ and $H$ for which $\rho(G\square H)-\rho(G)\rho(H)=b$.
\end{proposition}

\begin{proof}
Let $b=r(t-1)$, in which $r,t\geq 2$ are two integers. We consider the graph $G_{r,t}$ constructed as follows. We begin with a graph $G_r$ of order $r$. We next add $t$ pendant vertices to each vertex of $G_r$. In what follows we claim that
\begin{equation}\label{EQ100}
\rho(G_{r,t})=r.
\end{equation}

Let $S$ be a packing set of cardinality $\rho(G_{r,t})$. For any vertex $v$ of $G_{r}$, at most one vertex of $N[v]$ can belong to $S$. Thus, $\rho(G_{r,t})\le r$. By taking exactly one of the added pendant vertices to each vertex of $G_r$, we construct a packing set of $G_{r,t}$. This also means that $\rho(G_{r,t})\ge r$ and we have the equality (\ref{EQ100}).

We next give the exact value of the packing number of the Cartesian product of the graph $G_{r,t}$ and a complete graph $K_n$ on $n\geq t$ vertices. In fact, we claim that
\begin{equation}\label{EQ101}
\rho(G_{r,t}\square K_n)=rt.
\end{equation}

Let $V(K_n)=\{h_1,\dots,h_n\}$ and let $V(G_{r,t})=\{g_{1,0},g_{1,1},\dots,g_{1,t},\dots,g_{r,0},g_{r,1},\dots,g_{r,t}\}$, where $V(G_r)=\{g_{1,0},\dots,g_{r,0}\}$ and $N(g_{i,0})\setminus V(G_r)=\{g_{i,1},\dots,g_{i,t}\}$ for every $i\in\{1,\dots,r\}$. To prove our result, we consider the set
$$T=\{(g_{1,1},h_1),\dots,(g_{r,1},h_1),(g_{1,2},h_2),\dots,(g_{r,2},h_2),\dots,(g_{1,t},h_t),\dots,(g_{r,t},h_t)\}.$$

It can be noticed, by (\ref{Dist}), that the distance between any two vertices of $T$ is at least three. Thus, $T$ is a packing set of $G_{r,t}\square K_n$, and so, $\rho(G_{r,t}\square K_n)\geq|T|=rt$.

Let $B$ be a packing set of cardinality $\rho(G_{r,t}\square K_n)$. Consider the set of vertices of $G_{r,t}\square K_n$ in the matrix form $M=\{(g_{i,j},h_{k})\mid 1\leq i\leq r, 0\leq j\leq t\  \mbox{and}\ 1\leq k\leq n\}$. Hence, we can partition the set of rows of $M$ into the sets $M_{1,0},\cdots,M_{r,0}$, in which $M_{i,0}$ contains $t+1$ rows
$$((g_{i,0},h_{1}),\cdots,(g_{i,0},h_{n})),((g_{i,1},h_{1}),\cdots,(g_{i,1},h_{n})),\cdots,((g_{i,t},h_{1}),\cdots,(g_{i,t},h_{n})).$$
Since the subgraph induced by the vertices in $((g_{i,j},h_{1}),\cdots,(g_{i,j},h_{n}))$ for $0\leq j\leq t$, is a copy of the complete graph $K_{n}$, every row in $M_{i,0}$ has at most one vertex in $B$. Therefore, $|M_{i,0}\cap B|\leq t+1$ for each $1\leq i\leq r$. Suppose to the contrary that, $|M_{i,0}\cap B|=t+1$ for some $1\leq i\leq r$. Let $(g_{i,t},h_{j})$ be  the unique vertex of $((g_{i,t},h_{1}),\cdots,(g_{i,t},h_{n}))$ which belongs to $B$. This implies that $|N[(g_{i,0},h_{j})]\cap B|=2$, which is a contradiction. Thus, $|M_{i,0}\cap B|\leq t$ for each $1\leq i\leq r$, and so, $\rho(G_{r,t}\square K_n)=|B|\leq rt$. Therefore, we have the equality (\ref{EQ101}).

Since $\rho(K_n)=1$ and $\rho(G_{r,t})=r$, we have $\rho(G_{r,t}\square K_n)-\rho(G_{r,t})\rho(K_n)=r(t-1)=b$.
\end{proof}

For the open version of packings, we observe that the analogous inequality to that in (\ref{EQ21}) does not hold in general. To see this we consider for instance $G=H=P_{2}$. The definition of the packing number of graphs (based on a maximum value) is more feasible to find lower bounds than to find upper bounds for it. The main part of results known for the Cartesian product is an example of it. In contrast, in what follows we center our attention into giving an upper bound on $\rho(G\square H)$.

\begin{proposition}\label{Pro2}
For any graphs $G$ and $H$,
$$\rho(G\square H)\leq \min\{\rho(G)|V(H)|,\rho(H)|V(G)|\}.$$
This bound is sharp.
\end{proposition}

\begin{proof}
Let $V(H)=\{h_{1},\cdots,h_{|V(H)|}\}$. Clearly, $G\square H$ contains $|V(H)|$ disjoint $G$-layers. Now let $P$ be a maximum packing in $G\square H$. Hence, $P_{i}=P\cap G^{h_i}$ is a packing in $(G\Box H)[G^{h_i}]$, for each $1\leq i\leq|V(H)|$. Therefore, $|P_{i}|\leq \rho(G)$, and it follows that
$$\rho(G\square H)=|P|=\sum_{i=1}^{|V(H)|}|P_{i}|\leq \rho(G)|V(H)|.$$

We have $\rho(G\square H)\leq \rho(H)|V(G)|$, by a similar fashion. This shows the upper bound. That the bound is sharp can be seen by considering $\rho(P_{2}\square K_{m,n})=2$ when $m+n\geq 3$.
\end{proof}

Similarly to the inequality in Proposition \ref{Pro2}, we have
$$\rho_{o}(G\square H)\leq \min\{\rho_{o}(G)|V(H)|,\rho_{o}(H)|V(G)|\}.$$

In what follows we always assume that $|V(G)|,|V(H)|\geq 2$, for otherwise $G\square H\cong G$ or $G\square H\cong H$. If $G_{1},\ldots,G_{r}$ and $H_{1},\ldots,H_{s}$ are the components of $G$ and $H$, respectively, then $\rho_{0}(G\square H)=\sum_{i,j}\rho_{0}(G_{i}\square H_{j})$. So, we can assume that both $G$ and $H$ are connected.

We start with a lower bound on $\rho_{o}(G\square H)$. To this end, we let $\eta_{G}=\lceil(\mathrm{diam}(G)+1)/3\rceil$ and
\begin{equation}
\eta_{G}^{H}=\left \{
\begin{array}{lll}
1, & \mbox{if}\ G\neq K_{2}\ \mbox{and}\ \mathrm{diam}(H)\equiv2\ (\mbox{mod}\ 3), \vspace{1.5mm}\\
0, & \mbox{otherwise.}
\end{array}
\right .
\end{equation}

\begin{theorem}\label{Theo4}
For any graphs $G$ and $H$,
\begin{equation*}
\rho_{o}(G\square H)\geq \max\{\rho(G)\rho_{o}(H),\rho(H)\rho_{o}(G),\eta_{H}\rho_{o}(G)+\eta_{G}^{H},\eta_{G}\rho_{o}(H)+\eta_{H}^{G}\}.
\end{equation*}
\end{theorem}

\begin{proof}
Let $V(H)=\{v_1,\ldots,v_{|V(H)|}\}$ and let $P'$ be a maximum open packing in $G$ and $P'_{i}=P'\times \{v_i\}$. Consider a diametral path $P=v_{1}v_{2}\cdots v_{\mathrm{diam}(H)+1}$ in $H$, by renaming vertices if necessary. If $G=K_{2}$ or $\mathrm{diam}(H)\equiv0$ or $1$ (mod $3$), then
$$P_{1}=\cup_{i=1}^{\lfloor \mathrm{diam}(H)/3\rfloor+1}P'_{3i-2}$$
is an open packing in $G\square H$. Therefore, $\rho_{o}(G\square H)\geq|P_{1}|=\eta_{H}\rho_{o}(G)$. So, in what follows we suppose that $G\neq K_{2}$ and $\mathrm{diam}(H)\equiv 2$ (mod $3$). Let $u_{j}\in V(G)\setminus P'$ (note that such a vertex exists because $G\neq K_{2}$). Observe that $P_{2}=P_{1}\cup\{(u_{j},v_{\mathrm{diam}(H)+1})\}$ is an open packing in $G\square H$ of cardinality $\eta_{H}\rho_{o}(G)+1$. The above discussion results in $\rho_{o}(G\square H)\geq \eta_{H}\rho_{o}(G)+\eta_{G}^{H}$. Interchanging the roles of $G$ and $H$ we have $\rho_{o}(G\square H)\geq \eta_{G}\rho_{o}(H)+\eta_{H}^{G}$.

Now, let $P_{G}$ be a maximum open packing in $G$ and let $P_{H}$ be a maximum packing in $H$. Let $P=P_{G}\times P_{H}$. Suppose that there exists a vertex $(u,v)\in V(G\square H)$ with two neighbors $(u',v')$ and $(u'',v'')$ in $P$. Since $P_H$ is a packing of $H$, we have either $v'=v''$ or $d_H(v',v'')>2$. Because $(u,v)$ is a common neighbor of $(u',v')$ and $(u'',v'')$, we have $d_H(v',v'')\leq 2$. Therefore, $v'=v''$ follows. By the same reason, because $P_G$ is an open packing of $G$, we have that either $u'=u''$, or $u'u''\in E(G)$, or $d_G(u',u'')>2$. The first option cannot occur because $(u',v')\neq (u'',v'')$ and the last option is not possible since $(u',v')$ and $(u'',v'')$ have a common neighbor $(u,v)$. Therefore, $u'u''\in E(G)$. By the properties of the Cartesian product, the common neighbor $(u,v)$, of $(u',v')$ and $(u'',v'')$, must be also in the $G$-layer $G^{v'}$ because $v'=v''$, which is a contradiction with $P_G$ being an open packing of $G$ as different vertices $u',u''\in P_G$ have a common neighbor $u$. This shows that $P$ is an open packing in $G\square H$. So,
$$\rho_{o}(G\square H)\geq|P|=\rho_{o}(G)\rho(H).$$

Interchanging the roles of $G$ and $H$ we have $\rho_{o}(G\square H)\geq\rho_{o}(H)\rho(G)$, and this completes the proof.
\end{proof}

Since
\begin{equation}\label{EQ1}
\rho_{o}(K_{m}\square K_{n})=\left \{
\begin{array}{lll}
2, & \mbox{if}\ m=2\ \mbox{or}\ n=2, \vspace{1.5mm}\\
1, & \mbox{otherwise},
\end{array}
\right.
\end{equation}
we may assume that $\max\{\mathrm{diam}(G),\mathrm{diam}(H)\}\geq2$, where $\mathrm{diam}(G)$ and $\mathrm{diam}(H)$ stand for the diameters of $G$ and $H$, respectively.

From now on, we fix one factor, say $H$, to be a complete graph and try to bound the open packing number of $G\Box K_r$. In the next results, we follow a similar approach as for efficient open domination Cartesian products from \cite{KPRT}. Clearly, as $G\square K_{1}\cong G$, we only consider $r\ge 2$. For this we need to distinguish the cases when $r=2$ and when $r>2$. The reason for this is that for $r=2$, both vertices of a $K_r$-layer can be in an open packing, while this cannot occur when $r>2$, where at most one vertex from a $K_r$-layer can be in any open packing of $G\Box K_r$.

As an immediate consequence of Theorem \ref{Theo4}, we have $\rho_{o}(G\Box K_2)\geq2\rho(G)$, for any graph $G$. So, we next consider $\rho_{o}(G\Box K_r)$ for $r>2$. In this case, it is easy to see that every $K_r$-layer contains at most one vertex of any open packing of $G\Box K_r$.

Let $G$ be a graph and let $S$ be a $2$-independent set of $G$. We define a graph $G/S$ as follows. The vertex set $V(G/S)$ consists of components of $G[S]$. Two components $C$ and $C'$ of $G[S]$ are adjacent in $G/S$ whenever $d_{G}(C,C')=2$. We recall that $d_{G}(H_{1},H_{2})=\min\{d_{G}(v_{1},v_{2})\mid v_{1}\in V(H_{1})\ \mbox{and}\ v_{2}\in V(H_{2})\}$ for any subgraphs $H_{1}$ and $H_{2}$ of $G$, and also $\chi(G)$ is the chromatic number of $G$.

\begin{theorem}\label{theochi}
Let $G$ be a graph and $r>2$ an integer. Then,
$$\rho_{o}(G)\leq \rho_o(G\Box K_r)\leq \alpha_{2}(G).$$
If $G$ is triangle-free and $\chi(G/S)\leq r$, where $S$ is an $\alpha_{2}(G)$-set, then $\rho_o(G\Box K_r)=\alpha_{2}(G)$.
\end{theorem}
	
\begin{proof}
The lower bound follows from Theorem \ref{Theo4} because $\rho(K_n)=1$.

Suppose now that $\rho_o(G\Box K_r)>\alpha_{2}(G)$, let $V(K_r)=\{1,\ldots ,r\}$ and let $P$ be an $\rho_o(G\Box K_r)$-set. Clearly, every $K_r$-layer contains at most one vertex from $P$, because $r>2$. So, the projection $S'=p_G(P)$ contains $\rho_o(G\Box K_r)$ vertices of $G$. Notice that $S'$ is not a $2$-independent set because $\rho_o(G\Box K_r)>\alpha_{2}(G)$. Therefore there exists a vertex $x\in S'$ such that its degree is at least two in $G[S']$. Let $u$ and $v$ be neighbors of $x$ in $G[S']$ and let $(x,\ell),(u,i)$ and $(v,j)$ be the vertices of $P$ that project to $x,u$ and $v$, respectively. If $i=j$, then $(x,i)\in N_{G\Box K_r}((v,i))\cap N_{G\Box K_r}((u,i))$, which is a contradiction with $P$ being an open packing of $G\Box K_r$. Thus, we may assume that $i\neq j$. But then we have $\ell\neq i$ or $\ell\neq j$. Suppose without loss of generality that $\ell\neq i$. This implies that $(u,\ell)\in N_{G\Box K_r}((u,i))\cap N_{G\Box K_r}((x,\ell))$, which is a final contradiction. Hence $\rho_o(G\Box K_r)\leq \alpha_{2}(G)$ and the upper bound follows.

Let $G$ be a triangle-free graph, and let $S$ be an $\alpha_{2}(G)$-set for which $\chi(G/S)\leq r$. Also, let $V'_1,\ldots ,V'_k$, $k\leq r$, be the color classes of $G/S$. By $V_i$, $i\in \{1,\ldots,k\}$, we denote the set of vertices from $G$ that belong to $V'_i$. Clearly $\alpha_{2}(G)=|S|=\sum_{i=1}^k|V_i|$. We will show that the set
$$P=\{(v,i):v\in V_i, i\in\{1,\ldots ,k\}\}$$
forms an open packing of $G\Box K_r$ (notice that $P$ is well-defined as $k\leq r$). Clearly, $|P|=\alpha_{2}(G)$. Suppose first that $N_{G\Box K_r}((v,i))\cap N_{G\Box K_r}((u,i))\neq \emptyset$ for $v\neq u$. Therefore, there exists a vertex $x$ which is a common neighbor of $u$ and $v$ in $G$. On the other hand, $uv\notin E(G)$, for otherwise $u$, $v$ and $x$ would be on a triangle. By the definition of a $2$-independent set, it follows $x\notin S$. In consequence, $u$ and $v$ belong to two different components of $G[S]$, say to $C$ and $C'$, respectively. Vertices $C$ and $C'$ are adjacent in $G/S$, which yields a contradiction with a proper coloring of $G/S$ as $u,v\in V_{i}$ and with this $C,C'\in V'_{i}$. Thus, $N_{G\Box K_r}((v,i))\cap N_{G\Box K_r}((u,i))=\emptyset$ for any pair of different vertices $u,v\in V_i$. Suppose next that $N_{G\Box K_r}((v,i))\cap N_{G\Box K_r}((u,j))\neq \emptyset$ for $i\neq j$. Note that also $u\neq v$ by the definition of $P$. By the properties of the Cartesian product, $u$ must be adjacent  to $v$ in $G$. Therefore, $u,v\in S$ and they belong to one component of $G[S]$, a contradiction with $i\neq j$. Moreover, in this case we have $N_{G\Box K_r}((v,i))\cap N_{G\Box K_r}((u,j))= \emptyset$ and $P$ is an open packing of $G\Box K_r$. Therefore, $\rho_o(G\Box K_r)\geq\alpha_{2}(G)$, because $P$ is an open packing of cardinality $\alpha_2(G)$.
\end{proof}

The assumption which states that the graph $G$ is triangle-free is necessary for the equality $\rho_o(G\Box K_r)=\alpha_{2}(G)$ in Theorem \ref{theochi}. For instance, $\rho_o(K_{t}\Box K_r)=1$ while $\alpha_{2}(K_{t})=2$ for each integer $t\geq 3$.

A direct consequence of the last theorem holds for $r_0=\chi(G/S)$ and a set $S$ being an $\alpha_2(G)$-set of graph $G$.

\begin{corollary}\label{coreq}
Let $G$ be a triangle-free graph. There exists a positive integer $r_0$ such that $\rho_o(G\Box K_r)=\alpha_2(G)$ for every integer $r\geq r_0$.
\end{corollary}

With respect to the inequality $\rho_{o}(G\Box K_2)\geq 2\rho(G)$ (previously remarked) and Theorem \ref{theochi}, the question that remains concerns finding the value $\rho_o(G\Box K_r)$ whether $\chi(G/S)>r>2$ for an $\alpha_2(G)$-set $S$. One can show, by the same steps as in the proof of Theorem \ref{theochi}, that any proper, but partial, $r$-coloring of $G/S$ yields a packing of $G\Box K_r$. Similarly, every packing of $G\Box K_r$ yields a set of $G$ that is $2$-independent. However, it seems to be challenging to find a $2$-independent set that gives the maximum number of vertices in a partial $r$-coloring.


\section{Lexicographic, strong and direct products of graphs}

\ \ \ This section is divided into three parts. We first completely describe the packing and open packing numbers of the lexicographic product. We continue with bounds on the packing and open packing numbers of the strong product. In the last part we present lower bounds for the packing and open packing numbers of the direct product of graphs.

\subsection{Lexicographic product}

\ \ \ For two graphs $G$ and $H$, where $G$ is a connected graph of order at least two, it is a part of folklore that the distance between any two vertices $(g,h),(g',h')\in V(G\circ H)$ is given by
\begin{equation}\label{dist-lexico}
  d_{G\circ H}((g,h),(g',h'))=\left\{\begin{array}{ll}
                                       \min\{2,d_H(h,h')\}, & \mbox{if $g=g'$}, \\
                                       d_G(g,g'), & \mbox{if $g\ne g'$}.
                                     \end{array}
  \right.
\end{equation}
This formula does not hold anymore in the case of a disconnected graph $G$. In such a case the formula holds, if both $g$ and $g'$ are in the same component of $G$, when this component contains at least two vertices. If there exists a singleton $g$ in $G$, then we have $d_{G\circ H}((g,h),(g,h'))=d_H(h,h')$. This is the reason for the following notation. By $G^-$ we denote the graph obtained from $G$ by removing all the isolated vertices of $G$.

\begin{theorem}\label{pack-lexico}
Let $G$ and $H$ be any graphs. If $G$ has $i_G$ isolated vertices, then
$$\rho(G\circ H)=\rho(G)+i_G\rho(H)-i_G.$$
\end{theorem}

\begin{proof}
Let $P_G$ be a $\rho(G^-)$-set, let $P_H$ be a $\rho(H)$-set, and let $I$ be the set of all singletons of $G$. For a vertex $v\in V(H)$, we set $P=(P_G\times\{v\})\cup (I\times P_H)$. According to (\ref{dist-lexico}), and the paragraph after it, we can deduce that for any two vertices $(g,h),(g',h')\in P$, it follows that $d_{G\circ H}((g,h),(g',h'))\ge 3$. Thus, $P$ is a packing of $G\circ H$ and we have $\rho(G\circ H)\ge \rho(G)+i_G\rho(H)-i_G$.

On the other hand, let $P$ be a $\rho(G\circ H)$-set. Suppose first that $g\in V(G)$ is a singleton. The subgraph $(G\circ H)[^gH]$ is isomorphic to $H$ and contains at most $\rho(H)$ vertices of $P$. Hence, in $(G\circ H)[I\times V(H)]$, there exist at most $i_G\rho(H)$ vertices of $P$. Since the distance between any distinct vertices $(g,h),(g,h')\in V(G^-\circ H)$ is at most two, it clearly happens that for any $g\in V(G^-)$, $|P\cap (\{g\}\times V(H))|\le 1$. Moreover, for any two vertices $(g,h),(g',h')\in P$, where $g,g'\in V(G^-)$ (notice that it must be $g\ne g'$), it follows $d_G(g,g')=d_{G\circ H}((g,h),(g',h'))\ge 3$. Thus, the projection of $P-(I\times V(H))$ onto $G$ is a packing set in $G^-$. As a consequence, $\rho(G\circ H)=|P|\le |p_G(P-(I\times V(H)))|+i_G\rho(H)\le \rho(G)+i_G\rho(H)-i_G$, which leads to the equality.
\end{proof}

With respect to the open packing number, we can see that the lexicographic product behaves slightly different from that of the packing number.

\begin{theorem}
Let $G$ and $H$ be any graphs. If $G$ has $i_G$ isolated vertices, then
$$\rho_o(G\circ H)=\left\{\begin{array}{ll}
                            \rho_o(G)+i_G\rho_o(H)-i_G, & \mbox{if $H$ has an isolated vertex,} \\[0.2cm]
                            \rho(G)+i_G\rho_o(H)-i_G, & \mbox{otherwise.}
                          \end{array}
\right.$$
\end{theorem}

\begin{proof}
Suppose first that $h$ is an isolated vertex of $H$. Let $P_G$ be a $\rho_o(G^-)$-set, let $P_H$ be a $\rho_o(H)$-set, and let $I$ be the set of all isolated vertices of $G$. We set $P=(P_G\times\{h\})\cup (I\times P_H)$. If the open neighborhoods centered at distinct vertices $(g,h),(g',h)\in P_G\times \{h\}$ intersect, then we have a contradiction with $P_G$ being a $\rho_o(G^-)$-set. Similarly, if the open neighborhoods centered at distinct vertices $(g,h),(g,h')\in \{g\}\times P_H$ intersect for some $g\in I$, then we have a contradiction with $P_H$ being a $\rho_o(H)$-set. Now, assume $(g'',h'')$ belongs to the intersection of the open neighborhoods of $(g,h)\in P_G\times\{h\}$ and $(g',h')\in I\times P_H$. We note that $g''g\in E(G)$ and $g''=g'\in I$. This contradicts the fact that $g'$ is an isolated vertex in $G$. Thus, $P$ is an open packing of $G\circ H$ and we have $\rho_{o}(G\circ H)\ge \rho_o(G)+i_G\rho_o(H)-i_G$.

Conversely, let $P$ be a $\rho_o(G\circ H)$-set. Suppose first that $g\in V(G)$ is an isolated vertex. The subgraph $(G\circ H)[^gH]$ is isomorphic to $H$ and is a component of $H$. Therefore, $(G\circ H)[^gH]$ contains at most $\rho_o(H)$ vertices of $P$. Hence, in $(G\circ H)[I\times V(H)]$, there exist at most $i_G\rho_o(H)$ vertices of $P$. The distance between any two distinct vertices $(g,h),(g,h')\in V(G^-\circ H)$ is at most two and, if it is one, then they are on a common triangle. Hence, for any $g\in V(G^-)$, it follows $|P\cap (\{g\}\times V(H))|\le 1$. Therefore, we have $|p_G(P')|=|P'|$ for $P^-=P\cap (V(G^-)\times V(H))$. If $p_G(P')$ is not an open packing of $G^-$, then also $P'$ is not an open packing in $G^-\circ H$, which is a contradiction. Therefore, we have $\rho_o(G^-\circ H)\leq \rho_o(G^-)$. Joining both parts we obtain $\rho_o(G\circ H)\le \rho_o(G^-)+i_G\rho_o(H)=\rho_o(G)+i_G\rho_o(H)-i_G$ and the first equality holds.

Suppose now that $H$ has no isolated vertices. Let $P$ be an open packing in $G^-\circ H$ which is not a packing. Therefore, there exist two adjacent vertices $(g',h')$ and $(g'',h'')$ in $P$. Let $g'g''\in E(G)$. Since $h'$ is not an isolated vertex, there is a vertex $h'''\in V(H)$ adjacent to it. Hence $(g',h''')\in N_{G^-\circ H}((g',h'))\cap N_{G^-\circ H}((g'',h''))$, which is a contradiction. So, $g'=g''$ and $h'h''\in E(H)$. In addition, $(g''',h')\in N_{G^-\circ H}((g',h'))\cap N_{G^-\circ H}((g'',h''))$, in which $g'''$ is a neighbor of $g'$ in $G^-$, contradicting the fact that $P$ is an open packing. This shows that every open packing in $G^-\circ H$ is a packing, as well. So, $\rho_{o}(G^-\circ H)=\rho(G^-\circ H)$. Thus, $\rho_{o}(G^-\circ H)=\rho(G^-)$ follows from Theorem \ref{pack-lexico}, because $i_{G^-}=0$. It is also clear that $F=(G\circ H)[I\times V(H)]\cong i_GH$, or equivalently, $i_G$ copies of $H$. Clearly $\rho_o(F)=i_G\rho_o(H)$ and the result for $\rho_o(G\circ H)$ follows.
\end{proof}


\subsection{Strong product}

\ \ \ \ The strong product is a natural environment for the closed neighborhoods
\begin{equation}\label{neigh-strong}
  N_{G\boxtimes H}[(g,h)]=N_G[g]\times N_H[h].
\end{equation}

The distance between any two vertices $(g,h),(g',h')\in V(G\boxtimes H)$ is given by
\begin{equation}\label{dist-strong}
  d_{G\boxtimes H}((g,h),(g',h'))=\max\{d_G(g,g'),d_H(h,h')\}.
\end{equation}

The next result involves the fractional domination number $\gamma_f(G)$ of a graph $G$. Since this is not influential in this work, for the interested readers, we recommend \cite[p.360]{ImKl} for the definition and terminology on this issue.

\begin{theorem}\emph{(\cite[Theorem 28.16]{ImKl})}\label{strong}
If $G$ and $H$ are any graphs, then
$$\rho(G)\rho(H)\leq \rho(G\boxtimes H)\leq \min\{\rho(G)\gamma_f(H),\rho(H)\gamma_f(G)\}.$$
\end{theorem}

The lower bound from the last theorem follows easily from (\ref{neigh-strong}) because $P_G\times P_H$ is a packing of $G\boxtimes H$ for any packings $P_G$ and $P_H$ of $G$ and $H$, respectively. The upper bound is achieved by using some linear programming methods. However, we strongly believe that the lower bound gives always the exact result, but the proof of it seems to be very illusive. To underline such ideas we show the equality when one factor has a small diameter.

\begin{proposition}\label{strong1}
Let $G$ and $H$ be graphs. If ${\rm diam}(H)\leq 2$, then
$$\rho(G\boxtimes H)=\rho(G)\rho(H)=\rho(G).$$
\end{proposition}

\begin{proof}
Since ${\rm diam}(H)\leq 2$, we have $\rho(H)=1$ and the second equality is clear. By Theorem \ref{strong}, we know that $\rho(G\boxtimes H)\geq \rho(G)\rho(H)$. For the converse, let $P$ be a maximum packing of $G\boxtimes H$ and let $(g,h)$ be an arbitrary vertex from $P$. By (\ref{dist-strong}) we have $^gH\cap P=\{(g,h)\}$, because ${\rm diam}(H)\leq 2$. Similarly, we have $^{g'}H\cap P=\emptyset$ for any vertex $g'$ with $1\leq d_G(g,g')\leq 2$, again by (\ref{dist-strong}), and because ${\rm diam}(H)\leq 2$. Thus, the distance between any two vertices of $p_G(P)$ is at least three, and so $p_G(P)$ is a packing of $G$. Moreover, $p_G(P)$ has the same cardinality as $P$. Therefore, $\rho(G\boxtimes H)\leq\rho(G)=\rho(G)\rho(H)$ and the proof is completed.
\end{proof}

For the case of open packing sets the strong product behaves similarly as in the case of packings.

\begin{theorem}
For any graphs $G$ and $H$ with $i_G$ and $i_H$ isolated vertices, respectively,
$$\rho_o(G\boxtimes H)\geq\rho(G^-)\rho(H^-)+i_G\rho_o(H)+i_H\rho_o(G)-i_Gi_H$$
and
$$\rho_o(G\boxtimes H)\leq \min\{\rho(G^-)\gamma_f(H^-),\rho(H^-)\gamma_f(G^-)\}+i_G\rho_o(H)+i_H\rho_o(G)-i_Gi_H.$$
\end{theorem}

\begin{proof}
Let $P_G$ and $P_H$ denote a $\rho(G)$-set and a $\rho(H)$-set, respectively. Let $P^o_G$ and $P^o_H$ denote a $\rho_o(G)$-set and a $\rho_o(H)$-set, respectively. Finally, let $I_G$ and $I_H$ be the sets of isolated vertices from $G$ and $H$, respectively. We set
$$S=((P_G-I_G)\times (P_H-I_H))\cup (I_G\times (P^o_H-I_H)) \cup ((P^o_G-I_G)\times I_H)\cup(I_G\times I_H).$$
The four sets in this union are clearly disjoint and we have $|S|=\rho(G^-)\rho(H^-)+i_G(\rho_o(H)-i_H)+i_H(\rho_o(G)-i_G)+i_Gi_H=\rho(G^-)\rho(H^-)+i_G\rho_o(H)+i_H\rho_o(G)-i_Gi_H$. The set $S$ has the desired cardinality and, if we show that it is also an open packing of $G\boxtimes H$, then we have one inequality.

Every vertex from $A=I_G\times I_H$ is a singleton in $G\boxtimes H$, and so, at distance more than two away to any other vertex of $G\boxtimes H$. Let $g$ be an isolated vertex of $G$. The subgraph of $G\boxtimes H$ induced by $\{g\}\times (V(H)-I_H)$ is isomorphic to $H^-$. Moreover, every component of this subgraph is also a component of $G\boxtimes H$. Thus, the set $\{g\}\times (P_H^o-I_H)$ is a maximum open packing of these components. Therefore, also $B=I_G\times (P_H^o-I_H)$ is a maximum open packing of the subgraph induced by $I_G\times V(H^-)$, and $B$ is an open packing of $G\boxtimes H$. By symmetry, also the set $C=(P_G^o-I_G)\times I_H$ is an open packing of $G\boxtimes H$. Finally, $D=(P_G-I_G)\times (P_H-I_H)$ is an open packing of $(G-I_G)\boxtimes (H-I_H)$, because it is also a packing of $(G-I_G)\boxtimes (H-I_H)$ (see the remark after Theorem \ref{strong}). Because all four sets $A,B,C$ and $D$ belong to different components of $G\boxtimes H$, their union $S$ is an open packing of $G\boxtimes H$, and so $\rho_o(G\boxtimes H)\geq \rho(G^-)\rho(H^-)+i_G\rho_o(H)+i_H\rho_o(G)-i_Gi_H$.

For the upper bound we split $G\boxtimes H$ into four parts: $G^-\boxtimes H^-$, $G^-\boxtimes I_H$, $I_G\boxtimes H^-$ and $I_G\boxtimes I_H$. Let $P$ be a maximum packing of $G\boxtimes H$. Clearly, $I_G\times I_H\subseteq P$, since $I_G\times I_H$ consists of isolated vertices. The graph $I_G\boxtimes H^-$ is isomorphic to $i_G$ copies of $H^-$ and, if $P$ does not contain $i_G\rho_o(H^-)$ vertices from $I_G\boxtimes H^-$, then we immediately obtain a contradiction with the maximality of $P$. Similarly, $P$ contains exactly $i_H\rho_o(G^-)$ vertices from $G^-\boxtimes I_H$. Finally, for $G^-\boxtimes H^-$ we observe that every pair of adjacent vertices is contained in a triangle, because there are no isolated vertices in $G^-$, nor in $H^-$. Therefore, every open packing of $G^-\boxtimes H^-$ is also a packing of the same graph, and by Theorem \ref{strong}, we have $\rho_o(G^-\boxtimes H^-)=\rho(G^-\boxtimes H^-)\leq \min\{\rho(G^-)\gamma_f(H^-),\rho(H^-)\gamma_f(G^-)\}$. Everything together yields the desired upper bound which finishes the proof.
\end{proof}


\subsection{Direct product}

\ \ \ \ It is an easy observation that the open neighborhoods behave nicely in direct products
\begin{equation}\label{neigh-direct}
  N_{G\times H}((g,h))=N_G(g)\times N_H(h).
\end{equation}

For the distance formula in the direct product, see \cite{Kim}, we need first to define \textit{even distance} $d_G^e(u,v)$ and \textit{odd distance} $d_G^o(u,v)$ between two vertices $u$ and $v$ of $G$. The even distance $d_G^e(u,v)$ is the minimum even number of edges on a walk between $u$ and $v$, and if such a walk does not exist, then we set $d_G^e(u,v)=\infty$. Similarly, the odd distance $d_G^o(u,v)$ is the minimum odd number of edges on a walk between $u$ and $v$ and, if such a walk does not exist, then we set $d_G^o(u,v)=\infty$. In a connected bipartite graph $G$, clearly, only one of $d_G^e(u,v)$ and $d_G^o(u,v)$ is finite.

The distance between any two vertices $(g,h),(g',h')\in V(G\times H)$ is now given by
\begin{equation}\label{dist-direct}
  d_{G\times H}((g,h),(g',h'))=\min\{\max\{d_G^e(g,g'),d_H^e(h,h')\},\max\{d_G^o(g,g'),d_H^o(h,h')\}\}.
\end{equation}

Recall that by $G^-$ we denote the graph obtained from $G$ by removing all the isolated vertices.

\begin{theorem}
If $G$ and $H$ are graphs with $i_G$ and $i_H$ isolated vertices, respectively, then
$$\rho(G\times H)\geq\max\{\rho_o(G^-)\rho(H^-),\rho_o(H^-)\rho(G^-)\}+i_G|V(H)|+i_H|V(G)|-i_Gi_H.$$
\end{theorem}

\begin{proof}
Let $P_G$ and $P_H$ denote a $\rho(G)$-set and a $\rho(H)$-set, respectively, let $P^o_G$ and $P^o_H$ denote a $\rho_o(G)$-set and a $\rho_o(H)$-set, respectively, and let $I_G$ and $I_H$ be the sets of isolated vertices from $G$ and $H$, respectively. Notice that $I_G\subseteq P_G, I_H\subseteq P_H, I_G\subseteq P^o_G$ and $I_H\subseteq P^o_H$. Without lost of generality, we may assume that $\rho_o(G^-)\rho(H^-)\leq\rho_o(H^-)\rho(G^-)$. We set
$$S=((P_G-I_G)\times (P^o_H-I_H))\cup (I_G\times V(H)) \cup (V(G)\times I_H).$$
The first set of this union is clearly disjoint with the other two. On the other hand, we have $(I_G\times V(H)) \cap (V(G)\times I_H)=I_G\times I_H$ and we have $|S|=\rho(G^-)\rho_o(H^-)+i_G|V(H)|+i_H|V(G)|-i_Gi_H$.

Now, notice that  $(I_G\times V(H)) \cup (V(G)\times I_H)$ is the set of all singletons of $G\times H$, and can therefore be in any packing set. It remains to see that $(P_G-I_G)\times (P^o_H-I_H)$ represents a packing of $G^-\times H^-$. Let $(g,h)$ and $(g',h')$ be any distinct vertices from $(P_G-I_G)\times (P^o_H-I_H)$. If $g\neq g'$, then $d_G(g,g')\geq 3$, because $g$ and $g'$ belong to a packing of $G$. If $h\neq h'$, then either $d_H(h,h')=1$ or $d_H(h,h')\geq 3$, because $h$ and $h'$ belong to an open packing of $H$. Moreover, if $d_H(h,h')=1$, then $h$ and $h'$ do not belong to a common triangle, as $P^o_H$ is an open packing. Since $(g,h)\neq (g',h')$, we have $g\neq g'$ or $h\neq h'$. Suppose for instance $g\neq g'$. Clearly, $d^o_G(g,g')\geq 3$ and $d_G^e(g,g')> 3$, and so, we have $\max\{d_G^o(g,g'),d_H^o(h,h')\}\geq 3$ and $\max\{d_G^e(g,g'),d_H^e(h,h')\}>3$. By (\ref{dist-direct}) we have $d_{G\times H}((g,h),(g',h'))\geq 3$. Thus, we may assume that $g=g'$ and $h\neq h'$. Clearly, $d^o_G(g,g')\geq 3$ and $d_G^e(g,g')=0$. As mentioned, either $d_H(h,h')\geq 3$, or $h,h'$ are adjacent but not in the same triangle. Since $h,h'\in P^o_H$, we have $d_H(h,h')\neq 2$, and as $h$ and $h'$ are not in a common triangle, we have $d^e_H(h,h')> 2$. Altogether we have $\max\{d_G^o(g,g'),d_H^o(h,h')\}\geq 3$ and $\max\{d_G^e(g,g'),d_H^e(h,h')\}>2$. By (\ref{dist-direct}), we have $d_{G\times H}((g,h),(g',h'))\geq 3$, and so, $(P_G-I_G)\times (P^o_H-I_H)$ is a packing of $(G-I_G)\times (H-I_H)$. Therefore, $S$ is a packing set of $G\times H$ and the proof is completed.
\end{proof}

\begin{corollary}
If $G$ and $H$ are graphs without isolated vertices, then
$$\rho(G\times H)\geq\max\{\rho_o(G)\rho(H),\rho_o(H)\rho(G)\}.$$
\end{corollary}

The lower bound from last theorem is tight. In particular, it is well known that $G\times K_2\cong 2G$ for every bipartite graph $G$. Therefore,
$$\rho(G\times K_2)= 2\rho(G)=\max\{\rho_o(G),2\rho(G)\}=\max\{\rho_o(G)\rho(K_2),\rho_o(K_2)\rho(G)\}.$$

For the case of open packing sets, the direct product behaves nicer due to (\ref{neigh-direct}). Similarly to the case of packing number in the strong product, we expect equality to hold in the following theorem. However, again the proof for the upper bound seems very challenging, and we only present the lower bound.

\begin{theorem}
If $G$ and $H$ are graphs with $i_G$ and $i_H$ isolated vertices, respectively, then
$$\rho_o(G\times H)\geq\rho_o(G^-)\rho_o(H^-)+i_G|V(H)|+i_H|V(G)|-i_Gi_H.$$
\end{theorem}

\begin{proof}
Let $P_G$ and $P_H$ denote a $\rho_o(G)$-set and a $\rho_o(H)$-set, respectively, and let $I_G$ and $I_H$ be the sets of isolated vertices from $G$ and $H$, respectively. We set $$S=((P_G-I_G)\times (P_H-I_H))\cup (I_G\times V(H)) \cup (V(G)\times I_H).$$
The first set of this union is clearly disjoint with the other two. On the other hand, we have $(I_G\times V(H)) \cap (V(G)\times I_H)=I_G\times I_H$ and we have $|S|=\rho_o(G^-)\rho_o(H^-)+i_G|V(H)|+i_H|V(G)|-i_Gi_H$.

Notice that  $(I_G\times V(H)) \cup (V(G)\times I_H)$ is the set of all singletons of $G\times H$, and it is therefore contained in any maximum open packing. On the other hand, $(P_G-I_G)\times (P_H-I_H)$ represents an open packing of $(G-I_G)\times (H-I_H)$, by (\ref{neigh-direct}). Thus, $S$ is an open packing set of $G\times H$, and consequently, $\rho_o(G\times H)\geq \rho_o(G^-)\rho_o(H^-)+i_G|V(H)|+i_H|V(G)|-i_Gi_H$.
\end{proof}


\section{Rooted product graphs}

\ \ \ \ A \emph{rooted graph} is a graph in which one vertex is labeled in a special way to distinguish it from other vertices. The special vertex is called the \emph{root} of the graph. Let $G$ be a labeled graph on $n$ vertices. Let ${\cal H}$ be a sequence of $n$ rooted graphs $H_1,\ldots,H_n$. The \emph{rooted product graph} $G( {\cal H})$ is the graph obtained by identifying the root of $H_i$ with the $i^{th}$ vertex of $G$, see \cite{rooted-first}. We here consider the particular case of rooted product graphs where ${\cal H}$ consists of $n$  isomorphic rooted graphs \cite{Schwenk}. More formally, assuming that $V(G) = \{g_1,\ldots,g_n\}$ and that the root vertex of $H$ is $v$, we define the rooted product graph $G\circ_{v} H=(V,E)$, where $V=V(G)\times V(H)$ and
$$E=\displaystyle\bigcup_{i=1 }^n\{(g_i,h)(g_i,h'): \; hh'\in E(H)\}\cup \{(g_i,v)(g_j,v):\; g_ig_j\in E(G)\}.$$

Note that subgraphs induced by $H$-layers of $G\circ_v H$ are isomorphic to $H$. We next study the (open) packing number of rooted product graphs.

\begin{theorem}\label{pack-rooted}
Let $G$ be any graph of order $n$. If $H$ is any graph with root $v$, then
$$\rho(G\circ_v H)=\left\{\begin{array}{ll}
														\rho(G)+n(\rho(H)-1), &\mbox{$v\in P_H$ for every $\rho(H)$-set $P_H$},\\[0.2cm]
														n\rho(H), &\mbox{$v\notin P_H$ for some $\rho(H)$-set $P_H$}.
                          \end{array}
\right.$$
\end{theorem}

\begin{proof}
We distinguish two cases depending on the membership of $v$ to $\rho(H)$-sets.\\

\textit{Case 1.} Let $v$ belong to every $\rho(H)$-set $P_H$. Let $P'$ be any $\rho(G)$-set, and let now $P=(P'\times \{v\})\cup \left(V(G)\times (P_H\setminus \{v\})\right)$. We can easily see that $P$ is a packing set in $G\circ_v H$, which leads to $\rho(G\circ_v H) \ge \rho(G)+n(\rho(H)-1)$. On the other side, let $D$ be a $\rho(G\circ_v H)$-set. We note that the set $D_g=D\cap\: ^gH$ is a packing set in $(G\circ_v H)[^gH]$ for every $g\in V(G)$. Notice that $D_g$ is not a $\rho((G\circ_v H)[^gH])$-set for some $g\in V(G)$, since $v$ belongs to every $\rho(H)$-set (when $G$ contains at least one edge). Thus, we have that $|D\cap\: ^gH|=|D_g|\le \rho(H)-1$ if $(g,v)\notin D_g$, which means $|D\cap(^gH\setminus \{(g,v)\})|\le \rho(H)-1$. Also, if $(g,v)\in D_g$, then $|D\cap(^gH\setminus \{(g,v)\})|\le \rho(H)-1$ as well. In addition, $D\cap G^{v}$ is a packing set in $(G\circ_v H)[G^{v}]$ and hence $|D\cap G^v|\le \rho(G)$. As a consequence, we deduce the following
$$\rho(G\circ_v H)=|D|=|D\cap G^{v}|+\sum_{g\in V(G)}|D\cap(^gH\setminus \{(g,v)\})|\le \rho(G) + n(\rho(H)-1).$$
Therefore, $\rho(G\circ_v H)=\rho(G)+n(\rho(H)-1)$.\\

\textit{Case 2.} Suppose that there exists a $\rho(H)$-set $P_H$ not containing $v$. Let $P_H^g=\{g\}\times P_H$ for every $g\in V(G)$, and let $P''=\bigcup_{g\in V(G)}P_H^g$. It can be readily seen that $P''$ is a packing set in $G\circ_v H$, which means $\rho(G\circ_v H) \ge n\rho(H)$. On the other hand, let $P$ be a $\rho(G\circ_v H)$-set. It can be again easily observed that the set $P_g=P\cap\: ^gH$ is a packing set in $(G\circ_v H)[^gH]$ for every $g\in V(G)$. So, $\rho(H)=\rho((G\circ_v H)[^gH])\ge |P_g|$. As a consequence,
$$\rho(G\circ_v H)=|P|=\sum_{g\in V(G)}|P_g|\le \sum_{g\in V(G)}\rho(H) = n\rho(H),$$
which leads to $\rho(G\circ_v H)=n\rho(H)$.
\end{proof}

In \cite{BoPe} trees with unique maximum packing are described. Every leaf of such a tree must be contained in the unique maximum packing. If we choose a root to be a leaf in such a tree, then we end up in the first case of above theorem. Conversely, every neighbor $v$ of a leaf in a tree with unique maximum packing is not contained in the maximum packing. Therefore, if we choose $v$ to be the root of $H$ in such a tree, we end up in the second case of Theorem \ref{pack-rooted}.

By using similar techniques as in the theorem above, we can prove the following result for the case of open packing number. However, there are many different situations that must be remarked. For this we first introduce the following notation. Let $H$ be a graph and $v\in V(H)$. We say that the pair $(H,v)$ belongs to class $\cal{A}$ if $v\notin P_H$ for some $\rho_o(H)$-set $P_H$. On the other hand, $(H,v)\notin \cal{A}$ means that $v\in P_H$ for every $\rho_o(H)$-set $P_H$. We write $(H,v)\in \cal{B}$ if $\delta_{H[P_H]}(v)=0$ for some $\rho_o(H)$-set $P_H$. Clearly, $(H,v)\notin \cal{B}$ means that $\delta_{H[P_H]}(v)=1$ for every $\rho_o(H)$-set $P_H$. We also denote denote $H^*=H-N_H[v]$.

\begin{theorem}\label{open-pack-rooted}
Let $G$ be any graph of order $n$ with $i_G$ isolated vertices. If $H$ is any rooted graph with root $v$, then
$$\rho_o(G\circ_v H)=\left\{\begin{array}{ll}
						n\rho_o(H), &(H,v)\in \cal{A},\\[0.2cm]
						n(\rho_o(H)-1)+\rho_o(G), &\mbox{$(H,v)\notin \cal{A}$ and $(H,v)\in \cal{B}$},\\[0.2cm]
						n(\rho_o(H)-1)+\rho(G), &\mbox{$(H,v)\notin \cal{A}$ and $(H,v)\notin \cal{B}$ and $\rho_o(H^*)=\rho_o(H)-1$},\\[0.2cm]
						n(\rho_o(H)-1)+i_G, &\mbox{$(H,v)\notin \cal{A}$ and $(H,v)\notin \cal{B}$ and $\rho_o(H^*)\neq\rho_o(H)-1$}.
						
         \end{array}
\right.$$
\end{theorem}

\begin{proof}
Since every $H$-layer in $G\circ_v H$ is isomorphic to $H$, it follows that each $\rho_o(G\circ_v H)$-set intersects every $H$-layer in at most $\rho_o(H)$ vertices. Thus, $\rho_o(G\circ_v H)\leq n\rho_o(H)$. On the other hand, let $P_{H}$ be a $\rho_o(H)$-set. It is easy to see that $P=\bigcup_{g\in V(G)}(\{g\}\times(P_{H}\setminus\{v\}))$ is an open packing in $G\circ_v H$. Therefore, $n(\rho_o(H)-1)\leq|P|\leq \rho_o(G\circ_v H)$. We now distinguish two cases depending on the membership of $v$ to $\rho_o(H)$-sets.\\

\textit{Case 1.} $(H,v)\in \cal{A}$. There exists a $\rho_o(H)$-set not containing $v$. Let $P_{H}$ be such a set. We readily note that the set $P=\bigcup_{g\in V(G)}(\{g\}\times P_{H})$ is an open packing in $G\circ_v H$. Therefore, $n\rho_o(H)=|P|\leq \rho_o(G\circ_v H)$, which implies that $\rho_o(G\circ_v H)=n\rho_o(H)$ in this case.\\

\textit{Case 2.} $(H,v)\notin \cal{A}$. Thus $v$ belongs to every $\rho_o(H)$-set. In such a situation we must deal with two possibilities.\\

\textit{Subcase 2.1.} $(H,v)\in \cal{B}$. Hence there exists a $\rho_o(H)$-set $P_{H}$ for which $v$ is an isolated vertex of $H[P_{H}]$, which means that $\delta_{H[P_H]}(v)=0$. Let $P_{G}$ be a $\rho_o(G)$-set. It is easy to observe that $P'=\bigcup_{g\in V(G)}(\{g\}\times(P_{H}-\{v\}))\cup(\{v\}\times P_{G})$
is an open packing in $G\circ_v H$. Therefore, $\rho_o(G\circ_v H)\geq |P'|=n(\rho_o(H)-1)+\rho_o(G)$.

Suppose now that $\rho_o(G\circ_v H)>n(\rho_o(H)-1)+\rho_o(G)$ and let $P$ be a $\rho_{o}(G\circ_v H)$-set. Similarly to the proof of Theorem \ref{pack-rooted}, the set $P_g=P\cap\: ^gH$ is an open packing set of $(G\circ_v H)[^gH]$ for every $g\in V(G)$.
Moreover, $|P_g|=\rho_o(H)$, for at least $\rho_o(G)+1$ vertices $g\in V(G)$ because $|P|>n(\rho_o(H)-1)+\rho_o(G)$. Since $v$ belongs to every $\rho_o(H)$-set as $(H,v)\in \cal{A}$, it follows that at least $\rho_o(G)+1$ vertices $(g,v)\in V(G)\times\{v\}$ belong to $P$. This contradicts the fact that $P$ is an open packing in $G\circ_v H$. Therefore, $\rho_o(G\circ_v H)=|P'|=n(\rho_o(H)-1)+\rho_o(G)$.\\

\textit{Subcase 2.2.} $(H,v)\notin \cal{B}$. This means that $v$ has degree one in all subgraphs induced by $\rho_o(H)$-sets $P_{H}$, that is $\delta_{H[P_H]}(v)=1$. We first prove that $\rho_o(G\circ_v H)$ can be bounded from above by $n(\rho_o(H)-1)+\rho(G)$. Suppose to the contrary that $\rho_o(G\circ_v H)>n(\rho_o(H)-1)+\rho(G)$. Again, let $P$ be a $\rho_{o}(G\circ_v H)$-set and let $P_g=P\cap\: ^gH$ for any $g\in V(G)$. Note that $|P_g|=\rho_{o}(H)$, for at least $\rho(G)+1$ vertices $g\in V(G)$. Hence, at least $\rho(G)+1$ vertices, say $(g_{1},v),\cdots,(g_{\rho(G)+1},v)\in V(G)\times\{v\}$, belong to $P$. Since $\{(g_{1},v),\cdots,(g_{\rho(G)+1},v)\}$ is not a packing in $(G\circ_v H)[G^{v}]$ (note that $\{(g_{1},v),\cdots,(g_{\rho(G)+1},v)\}$ is an open packing in $(G\circ_v H)[G^{v}]$), there exists an edge $(g_{i},v)(g_{j},v)$ in $(G\circ_v H)[G^{v}]$ for some $1\leq i<j\leq \rho(G)$. This implies that both $(g_{i},v)$ and $(g_{j},v)$ have two neighbors in $P$, a contradiction. Therefore,
\begin{equation}\label{EEQ}
\rho_o(G\circ_v H)\leq n(\rho_o(H)-1)+\rho(G).
\end{equation}
From now on, we need to divide the reasoning into two more possibilities.\\

\textit{Subcase 2.2.1.} There exists an open packing $P'_{H^*}$ of cardinality $\rho_o(H^*)=\rho_o(H)-1$ in $H^*$. Let $P_{G}$ and $P^{o}_{H}$ be a $\rho(G)$-set and a $\rho_o(H)$-set, respectively. It is easily verified that
$$P''=\left(\bigcup_{g\in P_{G}}(\{g\}\times P^{o}_{H})\right)\cup \left(\bigcup_{g\in V(G)\setminus P_{G}}(\{g\}\times P'_{H^{*}})\right)$$
is an open packing in $G\circ_v H$. Therefore, $\rho_o(G\circ_v H)\geq|P''|=\rho(G)\rho_o(H)+(n-\rho(G))(\rho_o(H)-1)=n(\rho_o(H)-1)+\rho(G)$. We now have $\rho_o(G\circ_v H)=n(\rho_o(H)-1)+\rho(G)$, by (\ref{EEQ}).\\

\textit{Subcase 2.2.2.} $H$ does not satisfy the statement of Subcase 2.2.1. This means that every $\rho_o(H^*)$-set has cardinality $\rho_o(H^*)\neq\rho_o(H)-1$. In other words, every open packing of cardinality $\rho_o(H)-1$ of $H$ has at least one vertex in $N_H[v]$. Let $P$ be a $\rho_o(G\circ_v H)$-set. We have exactly $i_G$ components of $G\circ_v H$ isomorphic to $H$, where every component has exactly $\rho_o(H)$ vertices in $P$, and one component isomorphic to $G^-\circ_v H$ (recall that we obtain $G^-$ from $G$ by deleting all isolated vertices). In this case let $P^-=P\cap V(G^-\circ_v H)$ and let $P^-_g=P^-\cap\: ^gH$ for any $g\in V(G^-)$. We claim that $\rho_o(G^-\circ_v H)=(n-i_G)(\rho_o(H)-1)$. Suppose to the contrary that $\rho_o(G^-\circ_v H)>(n-i_G)(\rho_o(H)-1)$. Therefore, at least one subgraph $(G^-\circ_v H)[^{g}H]$, for some $g\in V(G^-)$, has exactly $\rho_o(H)$ vertices of $P^{-}$. Since $G^-$ has no isolated vertices, there exists a vertex $g'$ such that $gg'\in E(G^-)$. If $|P^-_{g'}|=\rho_o(H)$, then both $(g,v)$ and $(g',v)$ have two neighbors in $P^-$, which is a contradiction. Therefore, $|P^-_{g'}|\leq \rho_o(H)-1$ for all $g'\in N_{G^-}(g)$. Let now $g'$ be an arbitrary neighbor of $g$ in $G^-$. If $|P^-_{g'}|=\rho_o(H)-1$, then $N_{(G\circ_v H)[\:^{g'}H]}[(g',v)]\cap P^-_{g'}\neq\emptyset$. This implies that $|N_{G^-\circ_v H}(g,v)\cap P^-|\geq2$ or that $|N_{G^-\circ_v H}(g',v)\cap P^-|\geq2$, which is a contradiction. Therefore, $|P^-_{g'}|\leq \rho_o(H)-2$. The above argument guarantees that for every vertex $g\in V(G)$ such that $|P^-_{g}|=\rho_o(H)$, we have $|P^-_{g'}|\leq \rho_o(H)-2$ for all $g'\in N_{G^-}(g)$. This contradicts the assumption $\rho_o(G^-\circ_v H)>(n-i_G)(\rho_o(H)-1)$. Thus, $\rho_o(G^-\circ_v H)\leq (n-i_G)(\rho_o(H)-1)$. Therefore, it follows that $\rho_o(G^-\circ_v H)=(n-i_G)(\rho_o(H)-1)$ by using the corresponding inequality obtained at the first steps of the proof. We conclude this proof by the following computation
$$\rho_o(G\circ_v H)=\rho_o(G^-\circ_v H)+i_G\rho_o(H)=(n-i_G)(\rho_o(H)-1)+i_G\rho_o(H)=n(\rho_o(H)-1)+i_G.$$
\end{proof}

We end with four different examples that illustrate four cases of Theorem \ref{open-pack-rooted}. First let $H$ be a path $P_{4k+2}=v_1\ldots v_{4k+2}$ and let $v\in\{v_{4j-1},v_{4j}\}$ for some $j\in\{1,\ldots,k\}$. It is easy to see that there exists a unique $\rho_o(P_{4k+2})$-set of cardinality $2k+2$ not containing $v$. Hence $(P_{4k+2},v)\in\cal{A}$ and we have $\rho_o(G\circ_v P_{4k+2})=n\rho_o(P_{4k+2})=n(2k+2)$, where $n=|V(G)|$.

For the same reason (unique $\rho_o(P_{4k+2})$-set) we get for $H\cong P_{4k+2}$ and $v\in\{v_1,v_2,v_{4k+1},v_{4k+2}\}$ that $(H,v)\notin \cal{A}$ and $(H,v)\notin \cal{B}$ and $\rho_o(H^*)<\rho_o(H)-1$. Again by Theorem \ref{open-pack-rooted} we have $\rho_o(G\circ_v P_{4k+2})=n(\rho_o(P_{4k+2})-1)=n(2k+1)$, where $n=|V(G)|$.

Let now $H$ be a path $P_{4k+1}$ and let $v=v_1$. It is easy to see that $P=\{v_1\}\cup\{v_{4j},v_{4j+1}:1\leq j\leq k\}$ is a $\rho_o(P_{4k+1})$-set of cardinality $2k+1$ and that $v\in P$ such that $\delta_{H[P]}(v)=0$. Hence $(H,v)\notin \cal{A}$, $(H,v)\in \cal{B}$ and we have $\rho_o(G\circ_v P_{4k+1})=n(\rho_o(P_{4k+1})-1)+\rho_o(G)=2kn+\rho_o(G)$.

Finally, let $H$ be isomorphic to the complete graph $K_k$, $k\geq 4$, with one additional vertex $u$ adjacent to one vertex, say $v_1$, from $K_k$. Clearly, $u$ and $v_1$ form a unique $\rho_o(H)$-set and $\rho(H-\{u,v_1\})=1$. Thus, if we choose $v=u$, then the conditions $(H,v)\notin \cal{A}$ and $(H,v)\notin \cal{B}$ and $\rho_o(H^*)=\rho_o(H)-1$ are fulfilled. By the third option of Theorem \ref{open-pack-rooted} we have $\rho_o(G\circ_v H)=n(\rho_o(H)-1)+\rho(G)=n+\rho_o(G)$, where $n=|V(G)|$.


\end{document}